\documentclass{amsart}
\usepackage{amssymb}
\usepackage{amsfonts}
\usepackage{amssymb}
\usepackage{amsmath}
\usepackage{amsthm}
\usepackage{enumerate}
\usepackage{tabularx}
\usepackage{centernot}
\usepackage{mathtools}
\usepackage{stmaryrd}
\usepackage{amsthm,amssymb}
\usepackage{etoolbox}
\usepackage{tikz}
\usepackage{amssymb}
\usetikzlibrary{matrix}
\usepackage{tikz-cd}
\usepackage{tikz}
\usepackage{marginnote}
\definecolor{mygray}{gray}{0.85}
\usepackage[linecolor=black, backgroundcolor=mygray,colorinlistoftodos,prependcaption,textsize=small]{todonotes}
\usepackage{xargs}

\renewcommand{\leq}{\leqslant}
\renewcommand{\geq}{\geqslant}

\renewcommand{\trianglelefteq}{\trianglelefteqslant}

\makeatletter
\def\subsection{\@startsection{subsection}{3}%
  \z@{.5\linespacing\@plus.7\linespacing}{.3\linespacing}%
  {\bfseries\centering}}
\makeatother

\makeatletter
\def\subsubsection{\@startsection{subsubsection}{3}%
  \z@{.5\linespacing\@plus.7\linespacing}{.3\linespacing}%
  {\centering}}
\makeatother

\makeatletter
\def\myfnt{\ifx\protect\@typeset@protect\expandafter\footnote\else\expandafter\@gobble\fi}
\makeatother

\newtheorem{theorem}{Theorem}

\newtheorem{corollary}[theorem]{Corollary}
\newtheorem{definition}[theorem]{Definition}
\newtheorem{lemma}[theorem]{Lemma}
\newtheorem{proposition}[theorem]{Proposition}

\newtheorem{hypothesis}[theorem]{Hypothesis}

\newtheorem{fact}[theorem]{Fact}
\newtheorem{remark}[theorem]{Remark}
\newtheorem{notation}[theorem]{Notation}

\newtheorem*{ntheorem}{Theorem}
\newtheorem*{theorem1}{Theorem 1}

\newcounter{claimcounter}

\begin{document}

\dedicatory{Dedicated to the memory of Matti Rubin}

\begin{abstract} Let $\mathbf{K}$ be the class of countable structures $M$ with the strong small index property and locally finite algebraicity, and $\mathbf{K}_*$ the class of $M \in \mathbf{K}$ such that $acl_M(\{ a \}) = \{ a \}$ for every $a \in M$. For homogeneous $M \in \mathbf{K}$, we introduce what we call the expanded group of automorphisms of $M$, and show that it is second-order definable in $Aut(M)$. We use this to prove that for $M, N \in \mathbf{K}_*$, $Aut(M)$  and $Aut(N)$ are isomorphic as abstract groups if and only if $(Aut(M), M)$ and $(Aut(N), N)$ are isomorphic as permutation groups. In particular, we deduce that for $\aleph_0$-categorical structures the combination of strong small index property and no algebraicity implies reconstruction  up to bi-definability, in analogy with Rubin's well-known $\forall \exists$-interpretation technique of \cite{rubin}.
Finally, we show that every finite group can be realized as the outer automorphism group of $Aut(M)$ for some countable $\aleph_0$-categorical homogeneous structure $M$ with the strong small index property and no algebraicity.
\end{abstract}

\title[Reconstructing Structures with the SSIP up to Bi-Definability]{Reconstructing Structures with the Strong Small Index Property up to Bi-Definability}
\thanks{Partially supported by European Research Council grant 338821. No. 1109 on Shelah's publication list.}

\author{Gianluca Paolini}
\address{Einstein Institute of Mathematics,  The Hebrew University of Jerusalem, Israel}

\author{Saharon Shelah}
\address{Einstein Institute of Mathematics,  The Hebrew University of Jerusalem, Israel \and Department of Mathematics,  Rutgers University, U.S.A.}

\maketitle


\section{Introduction}

	Reconstruction theory deals with the problem of reconstruction of countable structures from their automorphism groups. The first degree of reconstruction that it is usually dealt with is the so-called {\em reconstruction up to bi-interpretability}. The second and stronger degree of reconstruction is known as {\em reconstruction up to bi-definability}. In group theoretic terms, the first degree of reconstruction corresponds to reconstruction of {\em topological group isomorphisms} from isomorphisms of abstract group, while the second degree of reconstruction corresponds to reconstruction of {\em permutation group isomorphisms} from isomorphisms of abstract group. Two independent techniques lead the scene in this field: the (strong) small index property (see e.g. \cite{lascar_hodges_shelah}) and Rubin's $\forall \exists$-interpretation \cite{rubin}. 

	On the reconstruction up to bi-interpretability side the cornerstones of the theory are the following two results:

	\begin{ntheorem}[Rubin \cite{rubin}] Let $M$ and $N$ be countable $\aleph_0$-categorical structures and suppose that $M$ has a $\forall \exists$-interpretation. Then $Aut(M) \cong Aut(N)$ if and only if $M$ and $N$ are bi-interpretable.
\end{ntheorem}

	\begin{ntheorem}[Lascar \cite{lascar}] Let $M$ and $N$ be countable $\aleph_0$-categorical structures and suppose that $M$ has the small index property. Then $Aut(M) \cong Aut(N)$ if and only if $M$ and $N$ are bi-interpretable.
\end{ntheorem}

	On the reconstruction up to bi-definability side, all the known results are based on the following theorem of Rubin:

	\begin{ntheorem}[Rubin \cite{rubin}] Let $M$ and $N$ be countable $\aleph_0$-categorical structures with no algebraicity and suppose that $M$ has a $\forall \exists$-interpretation. Then $Aut(M) \cong Aut(N)$ if and only if $M$ and $N$ are bi-definable.
\end{ntheorem}
	
	In particular, on the small index property side there is no result that pairs with the last cited result of Rubin. In this paper we fill this gap proving the following:
	
	\begin{theorem}\label{main_theorem} Let $\mathbf{K}_*$ be the class of countable structures $M$ satisfying:
	\begin{enumerate}[(1)]
	\item $M$ has the strong small index property;
	\item for every finite $A \subseteq M$, $acl_M(A)$ is finite;
	\item for every $a \in M$, $acl_M(\{ a \}) = \{ a \}$.
\end{enumerate}
Then for $M, N \in \mathbf{K}_*$, $Aut(M)$  and $Aut(N)$ are isomorphic as abstract groups if and only if $(Aut(M), M)$ and $(Aut(N), N)$ are isomorphic as permutation groups. Moreover, if $\pi: Aut(M) \cong Aut(N)$ is an abstract group isomorphism, then there is a bijection $f: M \rightarrow N$ witnessing that $(Aut(M), M)$ and $(Aut(N), N)$ are isomorphic as permutation groups and such that $\pi(\alpha) = f \alpha f^{-1}$.	
\end{theorem}

	Thus deducing an analog of Rubin's result on reconstruction up to bi-definability:

	\begin{corollary}\label{reconstr_bi-def} Let $M$ and $N$ be countable $\aleph_0$-categorical structures with the strong small index property and no algebraicity. Then $ Aut(M)$ and $Aut(N)$ are isomorphic as abstract groups if and only if $M$ and $N$ are bi-definable. Moreover, if $\pi: Aut(M) \cong Aut(N)$ is an abstract group isomorphism, then there is a bijection $f: M \rightarrow N$ witnessing the bi-definability of $M$ and $N$ such that $\pi(\alpha) = f \alpha f^{-1}$.	
\end{corollary}

	
	For a structure $M$ satisfying the conclusion of Theorem \ref{main_theorem} it is easy to determine the outer automorphism group of $Aut(M)$, in fact any $f \in Aut(Aut(M))$ is induced by a permutation of $M$. For example, as already noted by Rubin in \cite{rubin}, using this fact it is  easy to see that for $R_n$ the $n$-coloured random graph ($n \geq 2$) we have that $Out(Aut(R_n)) \cong Sym(n)$. Similarly, but in a different direction, one easily sees that for $M_n$ the $K_n$-free random graph ($n \geq 3$) we have that $Aut(M_n)$ is complete. We show here that in this setting any finite group can occur:

	\begin{theorem}\label{finite_grps} Let $K$ be a finite group. Then there exists a countable $\aleph_0$-categorical homogeneous structure $M$ with the strong small index property and no algebraicity such that $K \cong Out(Aut(M))$.
\end{theorem}

	Our main technical tool is what we call the {\em expanded group of automorphisms} of an homogeneous structure $M$ with the strong small index property and locally finite algebraicity. This powerful object encodes the combinatorics of $Aut(M)$-stabilizers of such a structure $M$, and it is a crucial ingredient of our proof of Theorem \ref{main_theorem}. In Theorem \ref{exp_auto_th} we show that the expanded group of automorphisms is second-order orbit-definable in $Aut(M)$ (cf. Definition \ref{orbit_def}), a fact of essential importance.

	The results of this paper pair with those of \cite{ssip_canonical_hom} and \cite{Sh_Pa_Hall}, where sufficient conditions for strong small index property are isolated and applied in the concrete case of the group of automorphisms of Hall's universal locally finite group. Finally, we would like to mention another recent result proved by us (see \cite{non_reconstr}) in this area, i.e. the following powerful non-reconstruction theorem: no algebraic or topological property of $Aut(M)$ can detect any form of stability of the countable structure $M$.

\section{The Expanded Group of Automorphisms}

	In this section we introduce the expanded group of automorphisms of $M$ (for certain $M$), and show that it is second-order definable in $Aut(M)$.
	
\smallskip

	Given a structure $M$ and $A \subseteq M$, and considering $Aut(M) = G$ in its natural action on $M$, we denote the pointwise (resp. setwise) stabilizer of $A$ under this action by $G_{(A)}$ (resp. $G_{\{ A \}}$). Also, we denote the subgroup relation by $\leq$.

\begin{definition}\label{algebraicity} Let $M$ be a structure and $G = Aut(M)$.
	\begin{enumerate}[(1)]
	\item We say that $a$ is algebraic (resp. definable) over $A \subseteq M$ in $M$ if the orbit of $a$ under $G_{(A)}$ is finite (resp. trivial). 
	\item The {\em algebraic closure} of $A \subseteq M$ in $M$, denoted as $acl_M(A)$, is the set of elements of $M$ which are algebraic over $A$.
	\item The {\em definable closure} of $A \subseteq M$ in $M$, denoted as $dcl_M(A)$, is the set of elements of $M$ which are definable over $A$.
\end{enumerate}
\end{definition}

\begin{definition} Let $M$ be a countable structure and $G = Aut(M)$.
	\begin{enumerate}[(1)]
	\item We say that $M$ (or $G$) has the {\em small index property} (SIP) if every subgroup of $Aut(M)$ of index less than $2^{\aleph_0}$ contains the pointwise stabilizer of a finite set $A \subseteq M$. 
	\item We say that $M$ (or $G$) has the {\em strong small index property} (SSIP) if every subgroup of $Aut(M)$ of index less than $2^{\aleph_0}$ lies between the pointwise  and the setwise stabilizer of a finite set $A \subseteq M$.
\end{enumerate}
\end{definition}

	\begin{hypothesis}\label{hyp} Throughout this section, let $M$ be a countable homogeneous structure with the strong small index property and locally finite algebraicity, i.e. for every finite $A \subseteq M$ we have $|acl_M(A)| < \omega$.
\end{hypothesis} 

	\begin{remark}\label{remark_cat} 
	Notice that all $\omega$-categorical structures have locally finite algebraicity.
\end{remark}

	\begin{notation} 
	\begin{enumerate}[(1)]
	\item We let $\mathbf{A}(M) = \{ acl_M(B) : B \subseteq_{fin} M \}$.
	\item We let $\mathbf{EA}(M) = \{ (K, L) : K \in \mathbf{A}(M) \text{ and } L \leq Aut(K) \}$.
\end{enumerate}
\end{notation}  

	\begin{definition} Let $(K, L) \in \mathbf{EA}(M)$, we define:
	$$ G_{(K, L)} = \{ f \in Aut(M) : f \restriction K \in L \}.$$
\end{definition}

Notice that if $L = \{ id_K \}$, then  $G_{(K, L)} = G_{(K)}$, i.e. it equals the pointwise stabilizer of $K$, and that if $L = Aut(K)$, then $G_{(K, L)} = G_{\{ K \}}$, i.e. it equals the setwise stabilizer of $K$. We then let:
$$\mathcal{PS}(M) = \{ G_{(K)} : K \in \mathbf{A}(M) \} \; \text{ and } \; \mathcal{SS}(M) =\{ G_{(K, L)} : (K, L) \in \mathbf{EA}(M) \}.$$
The crucial point is the following:
	
	\begin{lemma}\label{char_subgr_small_index} Let $\mathcal{G} = \{ H \leq G : [G:H] < 2^\omega \}$. Then $\mathcal{G} = \mathcal{SS}(M)$.
\end{lemma}

	\begin{proof} The containment from right to left is trivial. Let then $H \leq G$ with $ [G:H] < 2^\omega$. By the strong small index property, there is finite $K \subseteq M$ such that $G_{(K)} \leq H \leq G_{\{ K \}}$. It follows that $G_{(acl_M(K))} \leq H \leq G_{\{ acl_M(K) \}}$, and so without loss of generality we can assume that $K \in \mathbf{A}(M)$. First of all we claim that $G_{(K)} \trianglelefteq G_{\{ K \}}$. In fact, for $g \in G_{\{ K \}}$, $h \in G_{(K)}$ and $a \in K$, we have $ghg^{-1}(a) = gg^{-1}(a) = a$, since $g^{-1}(a) \in K$ and $h \in G_{(K)}$. Furthermore, for $g, h \in G_{\{ K \}}$, we have $g^{-1}h \in G_{(K)}$ iff $g \restriction K = h \restriction K$. Hence, the map $f: gG_{(K)} \mapsto g \restriction K$, for $g \in G_{\{ K \}}$, is such that: 
	\begin{equation}\tag{$\star$}\label{equation_label}
	f:G_{\{ K \}} / G_{(K)}  \cong  Aut(K),
\end{equation}
since every $f \in Aut(K)$ extends to an automorphism of $M$. Thus, by the fourth isomorphism theorem we have $H = G_{(K, L)}$ for $L = \{ f \restriction K : f \in H \}$.
\end{proof}

	\begin{proposition}\label{normality_prop} Let $H_1, H_2 \in \mathcal{SS}(M)$. The following conditions are equivalent:
	\begin{enumerate}[(1)]
	\item $H_1 \trianglelefteq H_2$ and $[H_2: H_1] < \omega$;
	\item there is $K \in \mathbf{A}(M)$ and $L_1 \trianglelefteq L_2 \leq Aut(K)$ such that $H_i = G_{(K, L_i)}$ for $i =1,2$.
\end{enumerate}
\end{proposition}

	\begin{proof} The proof of (2) implies (1) is immediate, since by the normality of $L_1$ in $L_2$ we have that, for $g \in G_{(K, L_2)}$ and $h \in G_{(K, L_1)}$, $g h g^{-1} \restriction K \in L_1$, while the fact that $[H_2: H_1] < \omega$ follows from the proof of Lemma \ref{char_subgr_small_index}. We show that (1) implies (2). By assumption, $H_i = G_{(K_i, L_i)}$ for $(K_i, L_i) \in \mathbf{EA}(M)$ ($i =1,2$). 
	\begin{enumerate}[$(*)_1$]
	\item $K_2 \subseteq K_1$.
\end{enumerate}
Suppose not, and let $a \in K_2 - K_1$ witness this. Then we can find $f \in G$ such that $f \restriction K_1 = id_{K_1}$ and $f(a) \not\in K_2$. It follows that $f \in H_1 - H_2$, a contradiction.
\begin{enumerate}[$(*)_2$]
	\item $K_1 \subseteq K_2$.
\end{enumerate}
Suppose not, and let $f_n \in G$, for $n < \omega$, such that $f_n \restriction K_2 = id_{K_2}$, and in addition $\{ f_n(K_1 - K_2) : n < \omega \}$ are pairwise disjoint. Then clearly, for every $n < \omega$, $f_n \in H_2$ and $\{ f_nH_1 : n < \omega \}$ are distinct, contradicting the assumption $[H_2: H_1] < \omega$.

\begin{enumerate}[$(*)_3$]
	\item $L_1 \leq L_2$.
\end{enumerate}
Suppose not, and let $h \in L_1 - L_2$. Then $h$ extends to an automorphism $f$ of $M$. Clearly $f \in H_1 - H_2$, a contradiction.

\begin{enumerate}[$(*)_4$]
	\item $L_1 \trianglelefteq L_2$.
\end{enumerate}
Suppose not, and let $g_i \in L_i$ ($i = 1, 2$) be such that $g_2g_1g_2^{-1} \not\in L_1$. Then $g_i$ extends to an automorphism $f_i$ of $M$ ($i = 1, 2$). Clearly $f_i \in H_i$ ($i = 1, 2$), and $f_2f_1f_2^{-1} \not\in H_1$, a contradiction.
\end{proof}

	\begin{proposition}\label{char_point_stab} Let $\mathcal{G} = \{ H \in \mathcal{SS}(M) : \text{ there is no } H' \in \mathcal{SS}(M), \text{with } H' \subsetneq H, H' \trianglelefteq H \text{ and } [H : H'] < \omega \}$. Then $\mathcal{PS}(M) = \mathcal{G}$.
\end{proposition}

	\begin{proof} First we show the containment from left to right. Let $H_2 \in \mathcal{PS}(M)$ and assume that there exists $H_1 \in \mathcal{SS}(M)$ such that $H_1 \subsetneq H_2, H_1 \trianglelefteq H_2 \text{ and } [H_2: H_1] < \omega$. By Proposition \ref{normality_prop}, $H_i = G_{(K_i, L_i)}$ for $(K_i, L_i) \in \mathbf{EA}(M)$ ($i =1,2$) and $K_1 = K = K_2$. Now, as $H_2 \in \mathcal{PS}(M)$, $L_2 = \{ id_K \}$. Hence, $L_1 = L_2$, and so $H_1 = H_2$, a contradiction. We now show the containment from right to left. Let $H \in \mathcal{G}$, then $H = G_{(K, L)}$ for $(K, L) \in \mathbf{EA}(M)$. If $L \neq \{ id_K \}$ then letting $H' = G_{(K, \{ id_K \})}$ we have $H' \subsetneq H$, $H' \trianglelefteq H$ and $[H : H'] < \omega$, a contradiction.
\end{proof}

	Let $\mathbf{L}(M)$ be a set of finite groups such that for every $K \in \mathbf{A}(M)$ there is a unique $L \in \mathbf{L}(M)$ such that $L \cong Aut(K)$.

\begin{proposition}\label{char_Lpoint_stab} Let $L \in \mathbf{L}(M)$ and $H \in \mathcal{SS}(M)$. The following conditions are equivalent:
	\begin{enumerate}[(1)]
	\item $H = G_{(K)} \in \mathcal{PS}(M)$ and $Aut(K) \cong L$;
	\item there is $H' \in \mathcal{SS}(M)$ such that $H \trianglelefteq H'$, $[H': H] < \omega$, $H'$ is maximal under these conditions and $H'/H \cong L$.
\end{enumerate}
\end{proposition}

	\begin{proof} Concerning the implication ``(1) implies (2)'', let $H' = G_{\{ K \}}$, then, by Proposition \ref{normality_prop} and equation (\ref{equation_label}) in the proof of Lemma \ref{char_subgr_small_index}, we have that $H'$ is as wanted. Concerning the implication ``(2) implies (1)'', if $H$ and $H'$ are as in (2), then, by Proposition \ref{normality_prop} and equation (\ref{equation_label}) in the proof of Lemma \ref{char_subgr_small_index}, it must be the case that $H' = G_{\{ K \}}$ and $H = G_{(K)}$ for some $K \in \mathbf{A}(M)$ such that $Aut(K) \cong L$.
\end{proof}

	\begin{definition}\label{def_expanded} We define the structure $ExAut(M)$, the {\em expanded group of automorphisms of $M$}, as follows:
	\begin{enumerate}[(1)]
	\item $ExAut(M)$ is a two-sorted structure;
	\item the first sort has set of elements $Aut(M) = G$;
	\item the second sort has set of elements $\mathbf{EA}(M)$;
	\item we identify $\{ (K, \{ id_K \}) : K \in \mathbf{A}(M) \}$ with $\mathbf{A}(M)$;
	\item the relations are:
	\begin{enumerate}
	\item $P_{\mathbf{A}(M)} = \{ K \in \mathbf{A}(M)\}$ (recalling the above identification);
	\item for $L \in \mathbf{L}(M)$, $P_{L(M)} = \{ K \in \mathbf{A}(M) : Aut(K) \cong L \}$;
	\item $\leq_{\mathbf{EA}(M)} \; = \{ ((K_1, L_1), (K_2, L_2)): (K_i, L_i) \in \mathbf{EA}(M) \; (i =1,2)\text{, } K_1 \leq K_2 \text{ and } L_2 \restriction K_1 \leq L_1 \}$;
	\item $\leq_{\mathbf{A}(M)} \; = \{ (K_1, K_2): K_i \in \mathbf{A}(M) \; (i =1,2)\text{ and } K_1 \leq K_2 \}$;
	\item $P^{min}_{\mathbf{A}(M)} = \{ K \in \mathbf{A}(M): acl(\emptyset) \neq K \in \mathbf{A}(M) \text{ is minimal in } (\mathbf{A}(M), \subseteq)\}$;
	\end{enumerate}
	\item the operations are:
	\begin{enumerate}[(f)]
	\item composition on $Aut(M)$;
	\end{enumerate}
	\begin{enumerate}[(g)]
	\item for $f \in Aut(M)$ and $K \in \mathbf{A}(M)$, $Op(f, K) = f(K)$;
	\end{enumerate}
	\begin{enumerate}[(h)]
	\item for $f \in Aut(M)$ and $(K_1, L_1) \in \mathbf{EA}(M)$, $Op(f, (K_1, L_1)) = (K_2, L_2)$ iff $f(K_1) = K_2$ and $L_2 = \{ f \restriction K_1 \pi f^{-1} \restriction K_2 : \pi \in L_1\}$.
	\end{enumerate}
	\end{enumerate}
\end{definition}

	
	\begin{definition}\label{orbit_def} We say that a set of subsets of a structure $N$ is second-order orbit-definable if it is preserved by automorphisms of $N$. We say that a structure $M$ is second-order orbit-definable in a structure $N$ if there is a injective map $\mathbf{j}$ mapping $\emptyset$-definable subsets of $M$ to second-order orbit-definable set of subsets of $N$.
\end{definition}

	\begin{theorem}\label{exp_auto_th} Let $M$ and $N$ be as in Hypothesis \ref{hyp}, and let $G = Aut(M)$. Then:
	\begin{enumerate}[(1)]
	\item The map $\mathbf{j}_M = \mathbf{j} : (f, (K, L)) \mapsto (\{ f \}, G_{(K, L)})$ witnesses second-order orbit-definability of $ExAut(M)$ in $Aut(M)$.
	\item Every $F: Aut(M) \cong Aut(N)$ has an extension $\hat{F}: ExAut(M) \cong ExAut(N)$.
	\end{enumerate}
\end{theorem}

	\begin{proof} We prove (1).
	
	\begin{enumerate}[$(*)_1$]
	\item The map $(f, (K, L)) \mapsto (\{ f \}, G_{(K, L)})$ is one-to-one.
\end{enumerate}
Suppose that $(K_1, L_1) \neq (K_2, L_2) \in \mathbf{EA}(M)$, we want to show that $G_{(K_1, L_1)} \neq G_{(K_2, L_2)}$. Suppose that $K_1 \neq K_2$. By symmetry, we can assume that $K_1 \not\subseteq K_2$. Then there is $f \in G$ such that $f \restriction K_2 = id_{K_2}$ and $f(K_1) \neq K_1$, since $K_2$ is algebraically closed. Thus, $f \in G_{(K_2, L_2)} - G_{(K_1, L_1)}$. Suppose now that $K_1 = K = K_2$ and $L_1 \neq L_2$. By symmetry, we can assume that $L_1 \not\subseteq L_2$. Let $g \in L_1 - L_2$, then $g$ extends to an automorphism $f$ of $M$. Thus, $f \in G_{(K, L_1)} - G_{(K, L_2)}$. Finally, notice that it is not possible that $\{ f \} = G_{(K, L)}$, and so we are done.

	\begin{enumerate}[$(*)_2$]
	\item The range $\mathbf{j}(\mathbf{EA}(M)) = \mathcal{SS}(M)$ is mapped onto itself by any $F \in Aut(G)$.
\end{enumerate}
By Lemma \ref{char_subgr_small_index}.

	\begin{enumerate}[$(*)_3$]
	\item The range $\mathbf{j}(P_{\mathbf{A}(M)})) = \mathcal{PS}(M)$ is mapped onto itself by any $F \in Aut(G)$.
\end{enumerate}
By Proposition \ref{char_point_stab}.

	\begin{enumerate}[$(*)_4$]
	\item For $L \in \mathbf{L}(M)$, the range $\mathbf{j}(P_{L(M)}) = \{ G_{(K)} : Aut(K) \cong L \}$ is mapped onto itself by any $F \in Aut(G)$.
\end{enumerate}
By Proposition \ref{char_Lpoint_stab}.


\begin{enumerate}[$(*)_5$]
	\item The range $\mathbf{j}(\leq_{\mathbf{EA}(M)}) = \{ (G_{(K_1, L_1)}, G_{(K_2, L_2)}) : G_{(K_1, L_1)} \supseteq G_{(K_2, L_2)}, (K_i, L_i) \in \mathbf{EA}(M) \text{ } (i=1, 2)\text{, } K_1 \leq K_2 \text{ and } L_2 \restriction K_1 \leq L_1 \}$ is preserved by any $F \in Aut(G)$.
\end{enumerate}
For $(K_i, L_i) \in \mathbf{EA}(M)$ ($i =1,2$) and $F \in Aut(G)$, obviously we have $\mathbf{j}(K_1, L_1) \supseteq \mathbf{j}(K_2, L_2)$ if and only if $F(\mathbf{j}(K_1, L_1)) \supseteq F(\mathbf{j}(K_2, L_2))$, since $F$ induces an automorphism of $(\mathcal{P}(Aut(G)), \subseteq)$.

\begin{enumerate}[$(*)_6$]
	\item The range $\mathbf{j}(\leq_{\mathbf{A}(M)}) = \{ (G_{(K_1)},  G_{(K_2)}) : G_{(K_1)} \supseteq G_{(K_2)},  K_1, K_2 \in \mathbf{A}(M), K_1 \leq K_2 \}$ is preserved by any $F \in Aut(G)$.
\end{enumerate}
As in $(*)_5$, i.e. any $F \in Aut(G)$ induces an automorphism of $(\mathcal{P}(Aut(G)), \subseteq)$.

\begin{enumerate}[$(*)_7$]
	\item The range $\mathbf{j}(P^{min}_{\mathbf{A}(M)}) = \{ H \in \mathcal{PS}(M) : G \neq H \text{ is maximal in } (\mathcal{PS}(M), \subseteq) \}$ is preserved by any $F \in Aut(G)$.
\end{enumerate}
As in $(*)_5$, i.e. any $F \in Aut(G)$ induces an automorphism of $(\mathcal{P}(Aut(G)), \subseteq)$.

\begin{enumerate}[$(*)_8$]
	\item For any $F \in Aut(G)$, $F(\{ gh \}) = F(\{g\})F(\{h\})$.
\end{enumerate}
Obvious.

\begin{enumerate}[$(*)_{9}$]
	\item $\mathbf{j}(Op(f, K)) = \mathbf{j}(f(K)) = G_{(f(K))} = fG_{(K)}f^{-1}$ and: $$F(\mathbf{j}((Op(f, K)))) = Op(F(f), F(\mathbf{j}(K))),$$ for any $F \in Aut(G)$.
\end{enumerate}
Observe that:
\[ \begin{array}{rcl}
	F(\mathbf{j}((Op(f, K)))) & = & F(fG_{(K)}f^{-1}) \\
					  	       & = & F(f)F(G_{(K)})(F(f))^{-1}\\
					           & = & F(f)(F(\mathbf{j}(K)) \\
					           & = & Op(F(f), F(\mathbf{j}(K))), 
\end{array}	\]
since by $(*)_3$ $\mathcal{PS}(M)$ is mapped onto itself by any $F \in Aut(G)$.

\begin{enumerate}[$(*)_{10}$]
	\item $\mathbf{j}(Op(f, (K_1, L_1))) = (fG_{(K_1)}f^{-1}, fG_{(K_1, L_1)}f^{-1})$ and: $$F(\mathbf{j}((Op(f, (K_1, L_1))))) = Op(F(f), F(\mathbf{j}((K_1, L_1)))),$$ for any $F \in Aut(G)$.
\end{enumerate}
Similar to $(*)_{9}$.

\smallskip

\noindent	This concludes the proof of (1). Finally, (2) follows directly from (1), in fact for $F: Aut(M) \cong Aut(N)$ letting $\hat{F} = \mathbf{j}_N^{-1}F\,\mathbf{j}_M$ we have $\hat{F}: ExAut(M) \cong ExAut(N)$.

\end{proof}

%
%


\section{Reconstruction and Outer Automorphisms}

	In this section we prove the theorems stated in the introduction.

\smallskip	
	
	\begin{definition}\label{minimal} Let $\mathbf{K}_*$ be the class of countable structures $M$ satisfying:
	\begin{enumerate}[(1)]
	\item $M$ has the strong small index property;
	\item for every finite $A \subseteq M$, $acl_M(A)$ is finite;
	\item for every $a \in M$, $acl_M(\{ a \}) = \{ a \}$;
\end{enumerate}
\end{definition}
	 
	 As in the previous section, we let $G = Aut(M)$. We denote $G_{(\{ a \})}$ simply as $G_{(a)}$. The crucial point in asking this additional condition is the following:
	 
	 \begin{proposition}\label{prop_for_recons} Let $M \in \mathbf{K}_*$ be homogeneous, and define:
$$\mathcal{M} = \{ G_{(a)} : a \in M \}.$$
Then $\mathbf{j}_M(P^{min}_{\mathbf{A}(M)}) = \mathcal{M}$ (recall Definition \ref{def_expanded} and Theorem \ref{exp_auto_th}).
\end{proposition}

	\begin{proof} Notice that by the third assumption in Definition \ref{minimal} we have that $P^{min}_{\mathbf{A}(M)} = \{ \{ a \} : a \in M \}$, and so directly by the definition of the interpretation $\mathbf{j}_M$ (cf. Theorem \ref{exp_auto_th}) we have that  $\mathbf{j}_M(P^{min}_{\mathbf{A}(M)}) = \mathcal{M}$.
\end{proof}

	We will use the suggestive notation $\mathcal{M} = \{ G_{(a)} : a \in M \}$ also below.
	
	\begin{definition} Let $M$ and $N$ be structures and consider $Aut(M)$ (resp. $Aut(N)$) as acting naturally on $M$ (resp. $N$). We say that $(Aut(M), M)$ and $(Aut(N), N)$ are isomorphic as permutation groups if there exists a bijection $f: M \rightarrow N$ such that the map $h \mapsto f h f^{-1}$ is an isomorphism from $Aut(M)$ onto $Aut(N)$.
\end{definition}

	We recall the statement of Theorem \ref{main_theorem} and prove it.

\begin{theorem1} Let $\mathbf{K}_*$ be the class of countable structures $M$ satisfying:
	\begin{enumerate}[(1)]
	\item $M$ has the strong small index property;
	\item for every finite $A \subseteq M$, $acl_M(A)$ is finite;
	\item for every $a \in M$, $acl_M(\{ a \}) = \{ a \}$.
\end{enumerate}
Then for $M, N \in \mathbf{K}_*$, $Aut(M)$  and $Aut(N)$ are isomorphic as abstract groups if and only if $(Aut(M), M)$ and $(Aut(N), N)$ are isomorphic as permutation groups.
\end{theorem1}

	\begin{proof}[Proof of Theorem \ref{main_theorem}] Let $M, N \in \mathbf{K}_*$, and suppose that $F: Aut(M) \cong Aut(N)$. Passing to canonical relational structures (cf. \cite[pg. 26]{cameron_oligo}), i.e. adding a relation symbol $R$ of arity $n$ for every $n$-ary $Aut(M)$-orbit $\Omega \subseteq M^n$, we can assume without loss of generality that $M$ and $N$ are homogeneous. Now, by Theorem \ref{exp_auto_th}(2), we have that the isomorphism $F$ induces the isomorphism:
	$$\hat{F} = \mathbf{j}_N^{-1}F\,\mathbf{j}_M : ExAut(M) \cong ExAut(N).$$
	In particular, $\hat{F}$ maps $P^{min}_{\mathbf{A}(M)}$ onto $P^{min}_{\mathbf{A}(N)}$. Furthermore, by Proposition \ref{prop_for_recons} we have:
	$$ \mathbf{j}_M(P^{min}_{\mathbf{A}(M)}) = \mathcal{M} \; \text{ and } \; \mathbf{j}_N(P^{min}_{\mathbf{A}(N)}) = \mathcal{N}.$$
Hence, $\hat{F}$ induces the bijection $f: M \rightarrow N$ such that (recall that $\hat{F} = \mathbf{j}_N^{-1}F\,\mathbf{j}_M$):
$$F(Aut(M)_{(a)}) = Aut(N)_{f(a)} \in \mathcal{N}.$$ Let $G: h \mapsto f h f^{-1}$, for $h \in Aut(M)$. We claim that $G = F$. Let in fact $h \in Aut(M)$, $a, b \in M$ and suppose that $F(h)(f(a)) = f(b)$. Then:
\[ \begin{array}{rcl}
	F(h)(f(a)) = f(b) & \Leftrightarrow & F(h)Aut(N)_{(f(a))}(F(h))^{-1} = Aut(N)_{(f(b))} \\
					  & \Leftrightarrow & hAut(M)_{(a)}h^{-1} = Aut(M)_{(b)}\\
					  & \Leftrightarrow & h(a) = b. 
\end{array}	\]
So, $fhf^{-1}(f(a)) = fh(a) = f(b)$, as wanted. Hence, $f: M \rightarrow N$ witnesses that $(Aut(M), M)$ and $(Aut(N), N)$ are isomorphic as permutation groups. Notice that the ``moreover part'' of the theorem is clear from the proof (since $G = F$).
\end{proof}

	\begin{definition} We say that two structures $M$ and $N$ are bi-definable if there is a bijection $f:M \rightarrow N$ such that for every $A \subseteq M^n$, $A$ is $\emptyset$-definable in $M$ if and only if $f(A)$ is $\emptyset$-definable in $N$.
\end{definition}

	\begin{fact}[\cite{rubin}, Proposition 1.3]\label{rubin_prop} Let $M$ and $N$ be countable $\aleph_0$-categorical structures. Then the following are equivalent:
	\begin{enumerate}[(1)]
	\item $(Aut(M), M) \cong (Aut(N), N)$;
	\item $M$ and $N$ are bi-definable.
	\end{enumerate}
\end{fact}

	\begin{proof}[Proof of Corollary \ref{reconstr_bi-def}] Let $M, N \in \mathbf{K}_*$, and suppose that $Aut(M) \cong Aut(N)$. As before, passing to canonical relational structures, we can assume without loss of generality that $M$ and $N$ are homogeneous. Furthermore, since $M$ and $N$ are $\aleph_0$-categorical, this passage preserves definability. Now, since $M$ and $N$ are $\aleph_0$-categorical and with no algebraicity, the conditions of Theorem \ref{main_theorem} are met (cf. Remark \ref{remark_cat}), and so we have that $(Aut(M), M) \cong (Aut(N), N)$ are isomorphic as permutation groups. Hence, by Fact \ref{rubin_prop}, we are done. Notice that the ``moreover part'' of the corollary is taken care of by the ``moreover part'' of Theorem~\ref{main_theorem}.
\end{proof}



	We now pass to the proof of Theorem \ref{finite_grps}. 
	
	\begin{fact}[Frucht's Theorem \cite{frucht}]\label{frucht} Every finite group is the group of automorphisms of a finite graph.
\end{fact}
	
	\begin{proof}[Proof of Theorem \ref{finite_grps}] Let $\Gamma$ be a finite graph on vertex set $\{ 0, ..., n-1 \}$ and 
	$$L_{\Gamma} = \{ P_{\ell} : \ell < n \} \cup \{ R_{\ell, k} : \ell < k < n \text{ and } \{\ell, k \} \in E_{\Gamma} \}$$ be such that the $P_{\ell}$ are unary predicates and the $R_{\ell, k}$ are binary relations. Let $\mathbf{K}_{\Gamma}$ be the class of finite $L_{\Gamma}$-models $M$ such that:
	\begin{enumerate}[(1)]
	\item $(P^M_{\ell} : \ell < n)$ is a partition of $M$;
	\item $R^M_{\ell, k}$ is a symmetric irreflexive relation on $P_{\ell} \times P_k$.
	\end{enumerate}
Notice that $\mathbf{K}_{\Gamma}$ is a free amalgamation class (cf. \cite[Definition 4]{ssip_canonical_hom}). Let $M_{\Gamma}$ be the corresponding countable homogeneous structure. By \cite[Corollary 2]{ssip_canonical_hom}, $M_{\Gamma}$ has the strong small index property, and, obviously, $M_{\Gamma}$ is $\aleph_0$-categorical and has no algebraicity. Using Corollary \ref{reconstr_bi-def} it is now easy to see that:
$$Aut(\Gamma) \cong Aut(Aut(M_{\Gamma}))/ Inn(Aut(M_{\Gamma})) = Out(Aut(M_{\Gamma})).$$
Thus, by Fact \ref{frucht} we are done.
\end{proof}

\end{document}